\def\year{2022}\relax
\documentclass[letterpaper]{article} 
\usepackage{aaai22}  
\usepackage{times}  
\usepackage{helvet}  
\usepackage{courier}  
\usepackage[hyphens]{url}  
\usepackage{graphicx} 
\usepackage{natbib}  
\usepackage{caption} 
\DeclareCaptionStyle{ruled}{labelfont=normalfont,labelsep=colon,strut=off} 
\usepackage{algorithm}
\usepackage{algorithmic}

\usepackage{xcolor}

\usepackage{amsmath,amssymb,amsfonts}
\usepackage{enumitem}
\usepackage{booktabs}
\usepackage{mathtools}				
\DeclareMathOperator{\rk}{rank}	
\DeclareMathOperator{\rank}{rank}	
\DeclareMathOperator{\tr}{tr}	    
\DeclareMathOperator{\vecc}{vec}	    
\DeclareMathOperator{\mat}{mat}	    
\DeclareMathOperator{\st}{s.t.}

\newcommand{\RR}{\mathbb R}
\DeclareMathOperator{\RIP}{RIP}
\DeclarePairedDelimiter{\abs}{\lvert}{\rvert}
\DeclarePairedDelimiter{\norm}{\lVert}{\rVert}
\DeclareMathOperator{\vect}{vec}

\usepackage{amsthm}
\newtheorem{theorem}{Theorem}

\newtheorem{lemma}{Lemma}
\theoremstyle{definition}
\newtheorem{definition}{Definition}

\theoremstyle{remark}

\usepackage{subcaption}

\title{Sharp Restricted Isometry Property Bounds for Low-rank Matrix Recovery Problems with Corrupted Measurements}
\author{
    Ziye Ma\equalcontrib \textsuperscript{\rm 1},
    Yingjie Bi\equalcontrib\textsuperscript{\rm 2},
    Javad Lavaei \textsuperscript{\rm 2},
    Somayeh Sojoudi \textsuperscript{\rm 1} \textsuperscript{\rm 3}
}
\affiliations{
    \textsuperscript{\rm 1}Department of Electrical Engineering and Computer Science\\
    \textsuperscript{\rm 2}Department of Industrial Engineering and Operations Research\\
    \textsuperscript{\rm 3}Department of Mechanical Engineering    

    University of California, Berkeley \\
    \texttt{\{ziyema,yingjiebi,lavaei,sojoudi\}@berkeley.edu}
}

\begin{document}
\maketitle

\begin{abstract}
In this paper, we study a general low-rank matrix recovery problem with linear measurements corrupted by some noise. The objective is to understand under what conditions on the restricted isometry property (RIP) of the problem local search methods can find the ground truth with a small error. By analyzing the landscape of the non-convex problem, we first propose a global guarantee on the maximum distance between an arbitrary local minimizer and the ground truth under the assumption that the RIP constant is smaller than $1/2$. We show that this distance shrinks to zero as the intensity of the noise reduces. Our new guarantee is sharp in terms of the RIP constant and is much stronger than the existing results. We then present a local guarantee for problems with an arbitrary RIP constant, which states that any local minimizer is either considerably close to the ground truth or far away from it. Next, we prove the strict saddle property, which guarantees the global convergence of the perturbed gradient descent method in polynomial time. The developed results demonstrate how the noise intensity and the RIP constant of the problem affect the landscape of the problem.
\end{abstract}

\section{Introduction}
Low-rank matrix recovery problems arise in various applications, such as matrix completion \citep{candes2009exact,recht2010guaranteed}, phase synchronization/retrieval \citep{singer2011angular,boumal2016nonconvex,shechtman2015phase}, robust PCA \citep{ge2017no}, and several others \citep{chen2018harnessing,chi2019nonconvex}. In this paper, we study a class of low-rank matrix recovery problems, where the goal is to recover a symmetric and positive semidefinite ground truth matrix $M^*$ with $\rank(M^*)=r$ from certain linear measurements corrupted by noise. This problem can be formulated as the following optimization problem:
\begin{equation}\label{eqn:main_canonical}
\begin{aligned}
\min_{M \in \RR^{n \times n}} \quad & \frac{1}{2}\norm{\mathcal A(M)-b+w}^2 \\
\st \quad & \rank(M) \leq r, \quad M \succeq 0.
\end{aligned}
\end{equation}
Here, $\mathcal A:\RR^{n \times n} \to \RR^m$ is a linear operator whose action on a matrix $M$ is given by
\[
\mathcal{A}(M)=[\langle A_1, M \rangle, \dots, \langle A_m, M \rangle]^T,
\]
where $A_1,\dots,A_m \in \mathbb{R}^{n \times n}$ are called sensing matrices. In addition, $b=\mathcal A(M^*)$ represents the perfect measurement on the ground truth $M^*$ and $w$ comes from an arbitrary probability distribution. Note that only the noisy measurement $b-w$ is available to the user, and indeed $b$ is unknown. In other words, from a problem-solving perspective, the random variable $w$ is hidden to the user, and it is explicitly modeled here only for the sake of analysis.

The modeling of the noise in this problem is critical for many practical applications in which the influence of the noise cannot be ignored. For example, state estimation is a major data analytic problem for the operation of power grids, which can be modeled as matrix sensing \citep{jin2019towards}. In this problem, each measurement comes from a physical device, and the noise in our formulation models not only the sensory noise but also mismatch between the true model of the system and the one used by a power operator, changes in the measurements due to cyber-attacks, mechanical faults, etc. In other words, by a noisy problem we mean a learning problem where a part of the data is wrong for various reasons, and as with this case, it is impossible to assume that the measurements are noise-free.

Although it may be possible to solve the problem \eqref{eqn:main_canonical} based on convex relaxations \citep{candes2009exact,recht2010guaranteed,candes2010power}, the computational complexity associated with solving a semidefinite program presents a major challenge for large-scale problems. A more scalable approach is to use the Burer--Monteiro factorization \citep{burer2003nonlinear} by expressing $M$ as $XX^T$ with $X \in \RR^{n \times r}$, which leads to the following equivalent formulation of the original problem \eqref{eqn:main_canonical}:
\begin{equation}\label{eqn:main_problem_1}
	\min_{X \in \mathbb{R}^{n \times r}} f(X) = \frac{1}{2} \|\mathcal{A}(XX^T) - b + w\|^2.
\end{equation}

The unconstrained problem \eqref{eqn:main_problem_1} is often solved by local search methods such as gradient descent. Since the objective function $f(X)$ in \eqref{eqn:main_problem_1} is non-convex, local search algorithms may converge to a local minimizer, leading to a suboptimal or plainly wrong solution. Hence, it is desirable to provide guarantees on the maximum distance between these local minimizers and the ground truth $M^*$. This problem will be addressed in this paper.

\subsection{Related Works}

The special noiseless case of the problem \eqref{eqn:main_problem_1} can be obtained by setting $w=0$. In this case, any solution $Z$ with $ZZ^T=M^*$ is a global minimizer of the problem \eqref{eqn:main_problem_1}. Many previous researches such as \citet{ge2017no,bhojanapalli2016global,PKCS2017,ZWYG2018,ZLTW2018,zhang2019sharp,zhang2020many,bi2020global,HLB2020,zhu2021global,Zhang2021-p} focused on proving that the problem has no spurious (non-global) local minimizers under the assumption of restricted isometry property (RIP). Moreover, as demonstrated in previous works such as \citet{ge2017no}, the developed techniques under the RIP condition can be adopted to show that other low-rank matrix recovery problems, such as matrix completion under incoherence condition and robust PCA problem, also have benign landscape. The RIP condition is equivalent to the restricted strongly convex and smooth property used in \citet{wang2017unified,park2018finding,zhu2021global}, and its formal definition is given below.

\begin{definition}
The linear operator $\mathcal{A}(\cdot): \mathbb{R}^{n \times n} \to \mathbb{R}^{m}$ is said to satisfy the $\delta$-RIP$_{2r}$ property for some constant $\delta \in [0,1)$ if the inequality
\[
(1-\delta)\|M\|_F^2 \leq \|\mathcal{A}(M)\|^2 \leq (1+\delta) \|M\|_F^2
\]
holds for all $M \in \mathbb{R}^{n \times n}$ with $\rk(M) \leq 2r$.
\end{definition}

In the recent paper by \citet{Zhang2021-p}, the author developed a sharp bound on the absence of spurious local minima for the noiseless case of problem \eqref{eqn:main_problem_1}, which says that the problem has no spurious local minima if the measurement operator $\mathcal A$ satisfies the $\delta$-$\RIP_{2r}$ property with $\delta<1/2$. This result is tight since there is a known counterexample \citep{zhang2018much} having spurious local minima under $\delta=1/2$.

For the noisy problem, the relation $X^*X^{*T}=M^*$ is unlikely to be satisfied, where $X^*$ denotes a global minimizer of problem \eqref{eqn:main_problem_1}. However, in this situation, $X^*X^{*T}$ should be close to the ground truth $M^*$ if the noise $w$ is small. As a generalization of the above-mentioned results for the noiseless problem, it is natural to study whether all local minimizers, including the global minimizers, are close to the ground truth $M^*$ under the RIP assumption. One such result is presented in \citet{bhojanapalli2016global} and given below.

\begin{theorem}[from Theorem~3.1 in \citet{bhojanapalli2016global}]
Suppose that $w \sim \mathcal{N}(0,\sigma_w^2 I_m)$ and $\mathcal{A}(\cdot)$ has the $\delta$-RIP$_{4r}$ property with $\delta<1/10$. Then, with probability at least $1-10/n^2$, any local minimizer $\hat X$ of problem \eqref{eqn:main_problem_1} satisfies the inequality
\[
		\|\hat X\hat X^T - M^*\|_F \leq 20 \sqrt{\frac{\log(n)}{m}} \sigma_w.
\]
	\label{thm:rip_noise}
\end{theorem}
Theorem 31 in \citet{ge2017no} further improves the above result by replacing the $\delta$-$\RIP_{4r}$ property with the $\delta$-$\RIP_{2r}$ property. \citet{li2020nonconvex} studies a similar noisy low-rank matrix recovery problem with $l_1$ norm.

As compared above, there is an evident gap between the state-of-the-art results for the noiseless and noisy problems. The result for the noiseless problem only requires the RIP constant $\delta<1/2$, but the one for the noisy problem requires $\delta<1/10$ no matter how small the noise is. This gap will be addressed in this paper by showing that a major generalization of Theorem~\ref{thm:rip_noise} holds for the noisy problem under the same RIP assumption as in the sharp bound for the noiseless problem.

\subsection{Notations}

In this paper, $I_n$ refers to the identity matrix of size $n \times n$. The notation $M \succeq 0$ means that $M$ is a symmetric and positive semidefinite matrix. $\sigma_i(M)$ denotes the $i$-th largest singular value of a matrix $M$, and $\lambda_i(M)$ denotes the $i$-th largest eigenvalue of $M$. $\norm{v}$ denotes the Euclidean norm of a vector $v$, while $\norm{M}_F$ and $\norm{M}_2$ denote the Frobenius norm and induced $l_2$ norm of a matrix $M$, respectively. $\langle A,B \rangle$ is defined to be $\tr(A^TB)$ for two matrices $A$ and $B$ of the same size. The Kronecker product between $A$ and $B$ is denoted as $A \otimes B$. For a matrix $M$, $\vecc(M)$ is the usual vectorization operation by stacking the columns of the matrix $M$ into a vector. For a vector $v \in \RR^{n^2}$, $\mat(v)$ converts $v$ to a square matrix and $\mat_S(v)$ converts $v$ to a symmetric matrix, i.e., $\mat(v)=M$ and $\mat_S(v) = (M+M^T)/2$, where $M \in \RR^{n \times n}$ is the unique matrix satisfying $v=\vecc(M)$. Finally, $\mathcal N(\mu,\mathbf\Sigma)$ refers to the multivariate Gaussian distribution with mean $\mu$ and covariance $\mathbf\Sigma$.

\section{Main Results}

We first present the global guarantee on the local minimizers of the problem \eqref{eqn:main_problem_1}. To simplify the notation, we use a matrix representation of the measurement operator $\mathcal A$ as follows:
\[
\mathbf{A} = [\vecc(A_1), \vecc(A_2), \dots ,\vecc(A_m)]^T \in \mathbb{R}^{m \times n^2}.
\]
Then, $\mathbf A\vecc(M)=\mathcal A(M)$ for every matrix $M \in \RR^{n \times n}$.

\begin{theorem}\label{thm:global}
Assume that the linear operator $\mathcal A$ satisfies the $\delta$-$\RIP_{2r}$ property with $\delta<1/2$. For every $\epsilon>0$, with probability at least $\mathbb P(\norm{\mathbf A^Tw} \leq \epsilon)$, either of the following two inequalities
\begin{subequations}
\begin{gather}
\begin{split}
(1-\delta)&\norm{\hat X\hat X^T-M^*}_F^2 \leq \epsilon \sqrt{r} \norm{\hat X\hat X^T-M^*}_F \\
&+4 \epsilon \sqrt{r} \norm{M^*}_F,
\end{split}\label{eq:ineq1_min} \\
\begin{split}
&\frac{2(1-2\delta)}{3(1+\delta)}\norm{\hat X\hat X^T-M^*}_F \leq 2\epsilon\sqrt r \\
&\hspace{1em}+2\sqrt{2\epsilon(1+\delta)}(\norm{\hat X\hat X^T-M^*}_F^{1/2}+\norm{M^*}_F^{1/2})
\end{split}\label{eq:ineq2_min}
\end{gather}
\end{subequations}
holds for every arbitrary local minimizer $\hat X \in \RR^{n \times r}$ of problem \eqref{eqn:main_problem_1}.
\end{theorem}

Note that two upper bounds on the distance $\norm{\hat X\hat X^T-M^*}_F$ can be obtained for any local minimizer $\hat X$ by solving the two quadratic-like inequalities \eqref{eq:ineq1_min} and \eqref{eq:ineq2_min}, and the larger bound needs to be used because only one of the two inequalities is guaranteed to hold. The explicit upper bound solved from Theorem~\ref{thm:global} is given in Appendix~\ref{app:remark}. The reason for the existence of two inequalities in Theorem~\ref{thm:global} is the split of its proof into two cases. The first case is when the $r$-th smallest singular value of $\hat X$ is small, and the second case is the opposite, which are respectively handled by Lemma~\ref{lem:case1} and Lemma~\ref{lem:case2}.

 The reason for the existence of two inequalities in Theorem~\ref{thm:global} is the split of its proof into two cases. The first case is when the $r$-th smallest singular value of $\hat X$ is small, and the second case is the opposite, which are respectively handled by Lemma~\ref{lem:case1} and Lemma~\ref{lem:case2}.

Theorem~\ref{thm:global} is a major extension of the existing sharp result stating that the noiseless problem has no spurious local minima under the same assumption of the $\delta$-$\RIP_{2r}$ property with $\delta<1/2$. The reason is that in the case when the noise $w$ is equal to zero, one can choose an arbitrarily small $\epsilon$ in Theorem~\ref{thm:global} to conclude from the inequalities \eqref{eq:ineq1_min} and \eqref{eq:ineq2_min} that $\hat X\hat X^T=M^*$ for every local minimizer $\hat X$. Moreover, when the RIP constant $\delta$ further decreases from $1/2$, the upper bound on $\norm{\hat X\hat X^T-M^*}_F$ will also decrease, which means that a local minimizer found by local search methods will be closer to the ground truth $M^*$. This suggests that the RIP condition is able to not only guarantee the absence of spurious local minima as shown in the previous literature but also mitigate the influence of the noise in the measurements.

Compared with the existing results such as Theorem~\ref{thm:rip_noise}, our new result has two advantages. First, by improving the RIP constant from $1/10$ to $1/2$, one can apply the results on the location of spurious local minima to a much broader class of problems, which can often help reduce the number of measurements. For example, in the case when the measurements are given by random Gaussian matrices, it is proven in \citet{candes2009exact} that to achieve the $\delta$-$\RIP_{2r}$ property the minimum number of measurements needed is in the order of $O(1/\delta^2)$. By improving the RIP constant in the bound, we can significantly reduce the number of measurements while still keeping the benign landscape. In applications such as learning for energy networks, there is a fundamental limit on the number of measurements that can be collected due to the physics of the problem \citep{JLSB2021}. Finding a better bound on RIP helps with addressing the issues with the number of measurements needed to reliably solve the problem. Second, Theorem~\ref{thm:rip_noise} is just about the probability of having all spurious solutions in a fixed ball around the ground truth of radius $O(\sigma_w)$ instead of balls of arbitrary radii, and this fixed ball could be a large one depending on whether the noise level $\sigma_w$ is fixed or scales with the problem. On the other hand, in Theorem~\ref{thm:global}, we consider the probability $\mathbb P(\norm{\mathbf A^Tw} \leq \epsilon)$ for any arbitrary value of $\epsilon$. By having a flexible $\epsilon$, our work not only improves the RIP constant but also allows computing the probability of having all spurious solutions in any given ball.

In the special case of rank $r=1$, the conditions \eqref{eq:ineq1_min} and \eqref{eq:ineq2_min} in Theorem~\ref{thm:global} can be substituted with a simpler condition as shown below. Its proof is highly similar to that of Lemma \ref{lem:case2}, and obviated in this paper for succinctness.

\begin{theorem}\label{thm:global_1}
Consider the case $r=1$ and assume that the linear operator $\mathcal A$ satisfies the $\delta$-$\RIP_2$ property with $\delta<1/2$. For every $\epsilon>0$, with probability at least $\mathbb P(\norm{\mathbf A^Tw} \leq \epsilon)$, every arbitrary local minimizer $\hat X \in \RR^{n \times r}$ of problem \eqref{eqn:main_problem_1} satisfies
\begin{equation}\label{eq:ineqrank1}
\norm{\hat X\hat X^T-M^*}_F \leq \frac{3(1+\sqrt 2)\epsilon(1+\delta)}{1-2\delta}.
\end{equation}
\end{theorem}

In the case when the RIP constant $\delta$ is not less than $1/2$, it is not possible to achieve a global guarantee similar to Theorem~\ref{thm:global} or Theorem~\ref{thm:global_1} since it is known that the problem may have a spurious solution even in the noiseless case. Instead, we turn to local guarantees by showing that every arbitrary local minimizer $\hat X$ of problem \eqref{eqn:main_problem_1} is either close to the ground truth $M^*$ or far away from it in terms of the distance $\norm{\hat X\hat X^T-M^*}_F$.

\begin{theorem}\label{thm:local_r}
Assume that the linear operator $\mathcal{A}$ satisfies the $\delta$-RIP$_{2r}$ property for some $\delta \in [0,1)$. Consider arbitrary constants $\epsilon>0$ and $\tau \in (0,2(\sqrt{2}-1))$ such that
\[
\delta<\sqrt{1-\frac{3+2\sqrt 2}{4}\tau^2}.
\]
Every arbitrary local minimizer $\hat X \in \RR^{n \times r}$ of problem \eqref{eqn:main_problem_1} satisfying
\begin{equation}\label{eq:localregion}
		 \|\hat X\hat X^T - M^*\|_F \leq \tau \lambda_r(M^*)
\end{equation}
will also satisfy
\begin{equation}\label{eq:deltalocalineq}
\norm{\hat X\hat X^T-M^*}_F \leq \frac{\epsilon(1+\delta)C(\tau,M^*)}{\sqrt{1-\frac{3+2\sqrt 2}{4}\tau^2}-\delta}
\end{equation}
with probability at least $\mathbb{P}(\norm{\mathbf A^Tw} \leq  \epsilon)$, where
\[
C(\tau,M^*)=\sqrt{\frac{2(\lambda_1(M^*)+\tau \lambda_r(M^*))}{(1-\tau)\lambda_r(M^*)} }.
\]
\end{theorem}

The upper bounds in \eqref{eq:localregion} and \eqref{eq:deltalocalineq} define an outer ball and an inner ball centered at the ground truth $M^*$. Theorem~\ref{thm:local_r} states that there is no local minimizer in the ring between the two balls. As $\epsilon$ approaches zero, the inner ball shrinks to the ground truth. This theorem shows that bad local minimizers are located outside the outer ball. Note that the problem could be highly non-convex when $\delta$ is close to 1, while this theorem shows a benign landscape in a local neighborhood of the solution. Furthermore, all the theorems in this section are applicable to arbitrary noise models since they make no explicit use of the probability distribution of the noise. The only required information is the probability $\mathbb{P}(\norm{\mathbf A^Tw} \leq  \epsilon)$, which can be computed or bounded when the probability distribution of the noise is given as illustrated in Section~\ref{sec:demo}.

The results presented above are all about the location of the local minimizers. They do not directly lead to the global convergence of local search methods with a fast convergence rate. To provide performance guarantees for local search methods, the next theorem establishes a stronger property for the landscape of the noisy problem that is usually called the strict saddle property in the literature, which essentially says that all approximate second-order critical points are close to the ground truth.

\begin{theorem}\label{thm:global_strictsaddle}
Assume that the linear operator $\mathcal A$ satisfies the $\delta$-$\RIP_{2r}$ property with $\delta<1/2$. For every $\epsilon>0$ and $\kappa \geq 0$, with probability at least $\mathbb P(\norm{\mathbf A^Tw} \leq \epsilon)$, either of the following two inequalities
\begin{subequations}
\begin{gather}
\begin{split}
(1&-\delta)\norm{\hat X\hat X^T-M^*}_F^2 \leq
\epsilon \sqrt{r} \norm{\hat X\hat X^T-M^*}_F \\
&+\frac{r^{1/4}\kappa}{2} \norm{\hat X\hat X^T-M^*}_F^{1/2}+\frac{r^{1/4}\kappa}{2} \norm{M^*}_F^{1/2}\\
&+ (4 \sqrt r \epsilon + \frac{5\sqrt r\kappa}{2})\norm{M^*}_F
\end{split}\label{eq:ineq1} \\
\begin{split}
&\frac{2(1-2\delta)}{3(1+\delta)}\norm{\hat X\hat X^T-M^*}_F \leq (2\epsilon+\kappa)\sqrt r+\sqrt{2\kappa(1+\delta)} \\
&\hspace{2em}+2\sqrt{2\epsilon(1+\delta)}(\norm{\hat X\hat X^T-M^*}_F^{1/2}+\norm{M^*}_F^{1/2})
\end{split}\label{eq:ineq2}
\end{gather}
\end{subequations}
holds for every matrix $\hat X \in \RR^{n \times r}$ satisfying
\[
\norm{\nabla f(\hat X)} \leq \kappa, \quad \nabla^2f(\hat X) \succeq -\kappa I_{nr}.
\]
\end{theorem}

Note that this property is not proven in the literature for $\delta<1/2$ even in the noiseless case, and thus our result generalizes the existing ones even in this scenario. On the other hand, it is proven by \citet{JGNK2017} that the perturbed gradient descent method with an arbitrary initialization will find a solution $\hat X$ satisfying the requirements in Theorem~\ref{thm:global_strictsaddle} with a high probability in $O(\mathrm{poly}(1/\kappa))$ number of iterations. By Theorem~\ref{thm:global_strictsaddle}, $\hat X\hat X^T$ will be close to the ground truth if $\epsilon$ and $\kappa$ are chosen to be relatively small.

Table~\ref{table:results_comparison} briefly summarizes our result compared with the existing literature.

\begin{table*}
\centering
\renewcommand*{\arraystretch}{1.5}
\begin{tabular}{lllll}
\toprule
\multicolumn{1}{l}{Paper} & \multicolumn{1}{l}{Noise} & \multicolumn{1}{l}{RIP Assumption} & \multicolumn{1}{l}{Rank} &\multicolumn{1}{l}{Convergence} \\\midrule
\citet{bhojanapalli2016global} & Isotropic Gaussian & $\delta < 1/10$ & rank $r$ & N/A                     \\
\citet{zhang2019sharp} & Noiseless & $\delta<1/2$ & rank 1 & N/A                     \\
\citet{Zhang2021-p} & Noiseless & $\delta<1/2$ & rank $r$ & N/A \\
Ours & Finite Variance & $\delta<1/2$ & rank $r$ & Polynomial                     \\\bottomrule
\end{tabular}
\caption{Comparison between our result and the existing literature.}\label{table:results_comparison}

\end{table*}

\section{Proofs of Main Results}\label{sec:proof}

Before presenting the proofs, we first compute the gradient and the Hessian of the objective function $f(\hat X)$ of the problem \eqref{eqn:main_problem_1}:
\begin{gather*}
\nabla f(\hat X)=\hat{\mathbf X}^T\mathbf A^T(\mathbf A\mathbf e+w), \\
\nabla^2f(\hat X)=2I_r \otimes \mat_S(\mathbf A^T(\mathbf A\mathbf e+w))+\hat{\mathbf X}^T\mathbf A^T\mathbf A\hat{\mathbf X},
\end{gather*}
where
\[
\mathbf e=\vecc(\hat X\hat X^T-M^*),
\]
and
$\hat{\mathbf X} \in \mathbb R^{n^2 \times nr}$ is the matrix satisfying
\[
	\hat{\mathbf X} \vecc(U) = \vecc(\hat XU^T +U\hat X^T), \quad \forall U \in \mathbb{R}^{n \times r}.
\]

The first step in the proofs is to derive necessary conditions for a matrix $\hat X \in \RR^{n \times r}$ to be an approximate second-order critical point, which depend on the linear operator $\mathcal A$, the noise $w \in \RR^m$, the solution $\hat X$, and the parameter $\kappa$ characterizing how close $\hat X$ is to a true second-order critical point.

\begin{lemma}\label{lem:necessary}
Given $\kappa \geq 0$, assume that $\hat X \in \RR^{n \times r}$ satisfies
\[
\norm{\nabla f(\hat X)} \leq \kappa, \quad \nabla^2f(\hat X) \succeq -\kappa I_{nr}.
\]
Then, it must satisfy the following inequalities:
\begin{subequations}
\begin{gather}
\norm{\hat{\mathbf X}^T\mathbf H\mathbf e} \leq 2\norm{\hat X}_2\norm{\mathbf A^Tw}+\kappa, \label{eq:foc} \\
2I_r \otimes \mat_S(\mathbf H\mathbf e)+\hat{\mathbf X}^T\mathbf H\hat{\mathbf X} \succeq -(2\norm{\mathbf A^Tw}+\kappa)I_{nr}, \label{eq:soc}
\end{gather}
\end{subequations}
where $\mathbf H=\mathbf A^T\mathbf A$.
\end{lemma}
\begin{proof}
To obtain condition \eqref{eq:foc}, notice that $\norm{\nabla f(\hat X)} \leq \kappa$ implies that
\begin{multline*}
\norm{\hat{\mathbf X}^T\mathbf H\mathbf e} \leq \norm{\hat{\mathbf X}^T\mathbf A^Tw}+\kappa \leq \norm{\hat{\mathbf X}}_2\norm{\mathbf A^Tw}+\kappa \\
\leq 2\norm{\hat X}_2\norm{\mathbf A^Tw}+\kappa,
\end{multline*}
in which the last inequality is due to
\[
\norm{\hat{\mathbf X}\vecc(U)}=\norm{\hat XU^T+U\hat X^T}_F \leq 2\norm{\hat X}_2\norm{U}_F,
\]
for every $U \in \RR^{n \times r}$. Similarly, $\nabla^2f(\hat X) \succeq -\kappa I_{nr}$ implies that
\[
2I_r \otimes \mat_S(\mathbf H\mathbf e)+\hat{\mathbf X}^T\mathbf H\hat{\mathbf X} \succeq -2I_r \otimes \mat_S(\mathbf A^Tw)-\kappa I_{nr}.
\]
On the other hand, the eigenvalues of $I_r \otimes \mat_S(\mathbf A^Tw)$ are the same as those of $\mat_S(\mathbf A^Tw)$, and each eigenvalue $\lambda_i(\mat_S(\mathbf A^Tw))$ of the latter matrix further satisfies
\[
\abs{\lambda_i(\mat_S(\mathbf A^Tw))} \leq \norm{\mat_S(\mathbf A^Tw)}_F \leq \norm{\mathbf A^Tw},
\]
which proves condition \eqref{eq:soc}.
\end{proof}

If $\hat X$ is a local minimizer of the problem \eqref{eqn:main_problem_1}, Lemma~\ref{lem:necessary} shows that $\hat X$ satisfies the inequalities \eqref{eq:foc} and \eqref{eq:soc} with $\kappa=0$. Similarly, Theorem~\ref{thm:global} can also be regarded as a special case of Theorem~\ref{thm:global_strictsaddle} with $\kappa=0$. The proofs of these two theorems consist of inspecting two cases. The following lemma deals with the first case in which $\hat X$ is an approximate second-order critical point with $\sigma_r(\hat X)$ being close to zero.

\begin{lemma}\label{lem:case1}
Given $\hat X \in \RR^{n \times r}$ and arbitrary constants $\epsilon>0$ and $\kappa \geq 0$, the inequalities
\[
\sigma_r(\hat X) \leq \sqrt\frac{\epsilon+\kappa}{1+\delta}, \; \norm{\nabla f(\hat X)} \leq \kappa, \; \nabla^2f(\hat X) \succeq -\kappa I_{nr}
\]
and
$\norm{\mathbf A^Tw} \leq \epsilon$ will together imply the inequality \eqref{eq:ineq1}.
\end{lemma}
\begin{proof}
Let $G=\mat_S(\mathbf H\mathbf e)$ and $u \in \RR^n$ be a unit eigenvector of $G$ corresponding to its minimum eigenvalue, i.e.,
\[
\norm{u}=1, \quad Gu=\lambda_{\min}(G)u.
\]
In addition, let $v \in \RR^r$ be a singular vector of $\hat X$ such that
\[
\norm{v}=1, \quad \norm{\hat Xv}=\sigma_r(\hat X).
\]
Let $\mathbf U=\vecc(uv^T)$. Then, $\norm{\mathbf U} \leq 1$ and \eqref{eq:soc} implies that
\begin{align}
-2\epsilon&-\kappa \leq 2\mathbf U^T(I_r \otimes \mat_S(\mathbf H\mathbf e))\mathbf U+\mathbf U^T\hat{\mathbf X}^T\mathbf H\hat{\mathbf X}\mathbf U \nonumber \\
&\leq 2\tr(vu^TGuv^T)+(1+\delta)\norm{\hat Xvu^T+uv^T\hat X^T}_F^2 \nonumber \\
&\leq 2\lambda_{\min}(G)+4(1+\delta)\sigma_r(\hat X)^2 \nonumber \\
&\leq 2\lambda_{\min}(G)+4\epsilon+4\kappa. \label{eq:mineig}
\end{align}
On the other hand,
\begin{align*}
(1-\delta)&\norm{\hat X\hat X^T-M^*}_F^2 \leq \mathbf e^T\mathbf H\mathbf e \\
&=\vecc(\hat X\hat X^T)^T\mathbf H\mathbf e-\vecc(M^*)^T\mathbf H\mathbf e \\
&=\frac{1}{2}\vecc(\hat X)^T\hat{\mathbf X}^T\mathbf H\mathbf e-\langle M^*,\mat_S(\mathbf H\mathbf e)\rangle \\
&\leq \frac{1}{2}\norm{\hat X}_2\norm{\hat{\mathbf X}^T\mathbf H\mathbf e}+\left(3\epsilon+\frac{5\kappa}{2}\right)\tr(M^*) \\
&\leq \epsilon\norm{\hat X}_2^2+\frac{\kappa}{2}\norm{\hat X}_2+\left(3\epsilon+\frac{5\kappa}{2}\right)\tr(M^*),
\end{align*}
in which the second last inequality is due to \eqref{eq:mineig} and the last inequality is due to \eqref{eq:foc}. Furthermore, the right-hand side of the above inequality can be relaxed as
\begin{align*}
\epsilon&\norm{\hat X}_2^2+\frac{\kappa}{2}\norm{\hat X}_2+\left(3\epsilon+\frac{5\kappa}{2}\right)\tr(M^*) \leq \epsilon \sqrt{r} \norm{\hat X\hat X^T}_F \\
&+\frac{r^{1/4}\kappa}{2} \norm{\hat X\hat X^T}_F^{1/2}+\left(3\epsilon+\frac{5\kappa}{2}\right)\sqrt r\norm{M^*}_F \\
&\leq \epsilon \sqrt{r} \norm{\hat X\hat X^T-M^*}_F+\frac{r^{1/4}\kappa}{2} \norm{\hat X\hat X^T-M^*}_F^{1/2} \\
&\hspace{1em}+ \left(4 \sqrt r \epsilon + \frac{5\sqrt r\kappa}{2}\right)\norm{M^*}_F+\frac{r^{1/4}\kappa}{2} \norm{M^*}_F^{1/2},
\end{align*}
which leads to the inequality \eqref{eq:ineq1}.
\end{proof}

The remaining case with
\[
\sigma_r(\hat X)>\sqrt\frac{\epsilon+\kappa}{1+\delta}
\]
will be handled in the following lemma using a different method.
\begin{lemma}\label{lem:case2}
Assume that the linear operator $\mathcal A$ satisfies the $\delta$-$\RIP_{2r}$ property with $\delta<1/2$. Given $\hat X \in \RR^{n \times r}$ and arbitrary constants $\epsilon>0$ and $\kappa \geq 0$, the inequalities
\[
\sigma_r(\hat X)>\sqrt\frac{\epsilon+\kappa}{1+\delta}, \; \norm{\nabla f(\hat X)} \leq \kappa, \; \nabla^2f(\hat X) \succeq -\kappa I_{nr}
\]
and
$\norm{\mathbf A^Tw} \leq \epsilon$ will together imply the inequality \eqref{eq:ineq2}.
\end{lemma}

The proofs of both Lemma~\ref{lem:case2} and the local guarantee in Theorem~\ref{thm:local_r} generalize the proof of the absence of spurious local minima for the noiseless problem in \citet{zhang2020many,Zhang2021-p}. Our innovation here is to develop new techniques to analyze approximate optimality conditions for the solutions because unlike the noiseless problem the local minimizers of the noisy one are only approximate second-order critical points of the distance function $\|\mathcal A(XX^T)-b\|^2$. For a fixed solution $\hat X$ and noise $w$, one can find an operator $\hat{\mathcal A}$ satisfying the $\delta$-$\RIP_{2r}$ property with the smallest possible $\delta$ such that $\hat X$ and $\hat{\mathcal A}$ satisfy the necessary conditions stated in Lemma~\ref{lem:necessary}. Let $\delta^*(\hat X)$ be the RIP constant of the found measurement operator $\hat{\mathcal A}$ in the worst-case scenario. Then, if $\hat X$ in Lemma~\ref{lem:case2} is a solution of the current problem with the linear operator $\mathcal A$ satisfying the $\delta$-$\RIP_{2r}$ property, it holds that $\delta \geq \delta^*(\hat X)$, which can further lead to an upper bound on the distance $\norm{\hat X\hat X^T-M^*}_F$.

To compute $\delta^*(\hat X)$ defined above, let $q=\mathbf A^Tw$ and solve the following optimization problem whose optimal value is $\delta^*(\hat X)$:
\begin{equation}\label{eqn:pre_lmi}
\begin{aligned}
\min_{\delta,\hat{\mathbf H}} \quad & \delta \\
\text{s.t.} \quad & \norm{\hat{\mathbf X}^T\hat{\mathbf H}\mathbf e} \leq 2\norm{\hat X}_2\norm{q}+\kappa, \\
& 2I_r \otimes \mat_S(\hat{\mathbf H}\mathbf e)+\hat{\mathbf X}^T\hat{\mathbf H}\hat{\mathbf X} \succeq -(2\norm{q}+\kappa)I_{nr}, \\
& \text{$\hat{\mathbf H}$ is symmetric and satisfies the $\delta$-$\RIP_{2r}$ property}.
\end{aligned}
\end{equation}
Note that a matrix $\hat{\mathbf H} \in \RR^{n^2 \times n^2}$ is said to satisfy the $\delta$-$\RIP_{2r}$ property if
\[
(1-\delta)\norm{\mathbf U}^2 \leq \mathbf U^T\hat{\mathbf H}\mathbf U \leq (1+\delta)\norm{\mathbf U}^2
\]
holds for every matrix $U \in \RR^{n \times n}$ with $\rk(U) \leq 2r$ and $\mathbf U=\vecc(U)$. Obviously, for a linear operator $\hat{\mathcal A}$, $\hat{\mathbf H}=\hat{\mathbf A}^T\hat{\mathbf A}$ satisfies the $\delta$-$\RIP_{2r}$ property if and only if $\hat{\mathbf A}$ satisfies the $\delta$-$\RIP_{2r}$ property.

However, since problem \eqref{eqn:pre_lmi} is non-convex due to the RIP constraint, we instead solve the following convex reformulation:
\begin{equation}\label{eqn:lmi}
\begin{aligned}
\min_{\delta,\hat{\mathbf H}} \quad & \delta \\
\st \quad & \norm{\hat{\mathbf X}^T\hat{\mathbf H}\mathbf e} \leq 2\norm{\hat X}_2\norm{q}+\kappa, \\
& 2I_r \otimes \mat_S(\hat{\mathbf H}\mathbf e)+\hat{\mathbf X}^T\hat{\mathbf H}\hat{\mathbf X} \succeq -(2\norm{q}+\kappa)I_{nr}, \\
& (1-\delta)I_{n^2} \preceq \hat{\mathbf H} \preceq (1+\delta)I_{n^2}.
\end{aligned}
\end{equation}
Lemma 14 in \citet{bi2020global} proves that problem \eqref{eqn:pre_lmi} and problem \eqref{eqn:lmi} have the same optimal value. The remaining step in the proof of Lemma~\ref{lem:case2} is to solve the optimization problem \eqref{eqn:lmi} for given $\hat X$, $q$ and $\kappa$. The complete proof of Lemma~\ref{lem:case2} is lengthy and deferred to Appendix~\ref{app:case2proof}. Finally, Theorem~\ref{thm:global} and Theorem~\ref{thm:global_strictsaddle} are direct consequences of Lemma~\ref{lem:case1} and Lemma~\ref{lem:case2}. The proof of Theorem~\ref{thm:global_1} is very similar to that of Lemma~\ref{lem:case2} and is also given in Appendix~\ref{app:case2proof}.

Now, we turn to the proof of the local guarantee in Theorem~\ref{thm:local_r}. The following existing result will be useful.

\begin{lemma}[from Lemma~14 in \citet{zhang2019sharp}]
	\label{lem:richard_14}
	Given $a,b \in \mathbb{R}^n$, the rank-2 matrix $ab^T + ba^T$ has two possibly nonzero eigenvalues
\[
\|a\|\|b\|(1+\cos \theta), \quad -\|a\|\|b\|(1-\cos \theta).
\]
Here, $\theta$ is the angle between $a$ and $b$.
\end{lemma}

\begin{proof}[Proof of Theorem~\ref{thm:local_r}]
First, we relax the optimization problem \eqref{eqn:lmi} by dropping the constraint related to the second-order necessary optimality condition. This gives rise to the optimization problem
\begin{equation}\label{eqn:lmi_stationary}
\begin{aligned}
\min_{\delta,\hat{\mathbf H}} \quad & \delta \\
\st \quad & \norm{\hat{\mathbf X}^T\hat{\mathbf H}\mathbf e} \leq 2\norm{\hat X}_2\norm{q}, \\
& (1-\delta)I_{n^2} \preceq \hat{\mathbf H} \preceq (1+\delta)I_{n^2}.
\end{aligned}
\end{equation}
To further simplify the problem \eqref{eqn:lmi_stationary}, one can replace its decision variable $\delta$ with $\eta$ and introduce the following optimization problem:
\begin{equation}\label{eqn:eta_stationary}
\begin{aligned}
\max_{\eta,\hat{\mathbf H}} \quad & \eta \\
\st \quad & \norm{\hat{\mathbf X}^T\hat{\mathbf H}\mathbf e} \leq 2\norm{\hat X}_2\norm{q}, \\
& \eta I_{n^2} \preceq \hat{\mathbf H} \preceq I_{n^2}.
\end{aligned}
\end{equation}
Given any feasible solution $(\delta,\hat{\mathbf H})$ to \eqref{eqn:lmi_stationary}, the tuple
\[
\left(\frac{1-\delta}{1+\delta},\frac{1}{1+\delta}\hat{\mathbf H}\right)
\]
is a feasible solution to problem \eqref{eqn:eta_stationary}. Therefore, if the optimal value of \eqref{eqn:lmi_stationary} is denoted as $\delta_f^*(\hat X)$ and the optimal value of \eqref{eqn:eta_stationary} is denoted as $\eta_f^*(\hat X)$, then it holds that
\begin{equation}\label{eqn:delta_eta}
\eta_f^*(\hat X) \geq \frac{1-\delta_f^*(\hat X)}{1+\delta_f^*(\hat X)} \geq \frac{1-\delta^*(\hat X)}{1+\delta^*(\hat X)} \geq \frac{1-\delta}{1+\delta},
\end{equation}
in which the last inequality is implied by $\delta \geq \delta^*(\hat X)$ as shown above. To prove the inequality \eqref{eq:deltalocalineq}, we need to bound $\eta_f^*(\hat X)$ from above, which can be achieved by finding a feasible solution to the dual problem of \eqref{eqn:eta_stationary} given below:
\begin{equation}\label{eqn:eta_dual}
\begin{aligned}
\min_{U_1,U_2,G,\lambda,y} \quad & \tr(U_2)+4\norm{\hat X}^2_2\norm{q}^2 \lambda+\tr(G) \\
\st \quad & \tr(U_1)=1, \\
& (\hat{\mathbf X}y)\mathbf e^T+\mathbf e(\hat{\mathbf X}y)^T=U_1-U_2, \\
& \begin{bmatrix}
G & -y \\
-y^T & \lambda
\end{bmatrix} \succeq 0, \\
& U_1 \succeq 0, \quad U_2 \succeq 0.
\end{aligned}
\end{equation}

For any matrix $\hat X \in \RR^{n \times r}$ satisfying $\norm{\hat X\hat X^T-M^*}_F \leq \tau \lambda_r(M^*)$, we have $\hat X \neq 0$, and it has been shown in the proof of Lemma 19 in \citet{bi2020global} that there exists $y \neq 0$ satisfying the inequalities
\begin{subequations}
\begin{gather}
\norm{\hat{\mathbf X}y}^2 \geq 2\lambda_r(\hat X\hat X^T)\norm{y}^2, \label{eq:yineq1} \\
\cos\theta \geq \sqrt{1-\frac{3+2\sqrt 2}{4}\tau^2}, \label{eq:yineq2}
\end{gather}
\end{subequations}
where $\theta$ is the angle between $\hat{\mathbf X} y$ and $\mathbf e$. Note that \eqref{eq:yineq2} holds only in the exact-parameterized regime, i.e., the case with $\rk(M^*) = r$, since the derivation of \eqref{eq:yineq2} utilized Lemma~5.4 of \citet{TBSS2016}, which does not hold when $\rk(M^*) < r$. This is the main reason why our result cannot be directly generalized to the overparameterized case. Now, define
\[
M=(\hat{\mathbf X} y)\mathbf e^T+\mathbf e(\hat{\mathbf X} y)^T,
\]
and then decompose $M$ as $M = [M]_+ - [M]_-$ with $[M]_+ \succeq 0$ and $[M]_- \succeq 0$. Then, it is easy to verify that $(U_1^*,U_2^*,G^*,\lambda^*,y^*)$ defined as
\begin{gather*}
y^* = \frac{y}{\tr([M]_+)}, \quad U_1^* =  \frac{[M]_+}{\tr([M]_+)}, \quad U_2^* = \frac{[M]_-}{\tr([M]_+)}, \\
G^* = \frac{y^* (y^*)^T}{ \lambda^*}, \quad \lambda^* = \frac{\| y^* \|}{2\norm{ \hat X}_2\norm{q}}
\end{gather*}
forms a feasible solution to the dual problem \eqref{eqn:eta_dual} with the objective value
\begin{equation}\label{eq:dualobj}
\frac{\tr([M]_-)+ 4\norm{\hat X}_2\norm{q} \| y\|}{\tr([M]_+)}.
\end{equation}

Furthermore, $\rank(M^*)=r$ implies that $\lambda_r(M^*) > 0$. By the Wielandt--Hoffman theorem,
\begin{align*}
	&| \lambda_r(\hat X\hat X^T) - \lambda_r(M^*)| \leq \|\hat X\hat X^T-M^*\|_F \leq \tau \lambda_r(M^*), \\
	&| \lambda_1(\hat X\hat X^T) - \lambda_1(M^*)| \leq \|\hat X\hat X^T-M^*\|_F \leq \tau \lambda_r(M^*).
\end{align*}
Thus, using the above two inequalities and inequality \eqref{eq:yineq1}, we have
\begin{multline}\label{eq:yineq3}
	\frac{2 \norm{\hat X}_2 \|y\|}{\|\hat{\mathbf X} y\|} \leq \frac{2 \norm{\hat X}_2}{\sqrt{2 \lambda_r(\hat X\hat X^T)}} \\
\leq \sqrt{\frac{2(\lambda_1(M^*)+\tau \lambda_r(M^*))}{(1-\tau)\lambda_r(M^*)} }=C(\tau,M^*).
\end{multline}
Next, according to Lemma~\ref{lem:richard_14}, one can write
\begin{gather*}
\tr([M]_+) = \|\hat{\mathbf X} y\| \|\mathbf e\| (1+ \cos \theta), \\
\tr([M]_-) = \|\hat{\mathbf X} y\| \|\mathbf e\| (1- \cos \theta).
\end{gather*}
Substituting the above two equations and \eqref{eq:yineq3} into the dual objective value \eqref{eq:dualobj}, one can obtain
\[
\eta_f^*(\hat X) \leq \frac{1- \cos \theta+2C(\tau,M^*)\norm{q}/\norm{\mathbf e}}{1+\cos \theta},
\]
which together with \eqref{eqn:delta_eta} implies that
\[
\norm{\mathbf e} \leq (1+\delta)C(\tau,M^*)\norm{q}(\cos\theta-\delta)^{-1}.
\]
The inequality \eqref{eq:deltalocalineq} can then be proved by combining the above inequality and \eqref{eq:yineq2} under the probabilistic event that $\norm{q} \leq \epsilon$.
\end{proof}

\begin{figure*}[t]
    \centering
    \begin{subfigure}{6.5cm}
    	    \includegraphics[width=\linewidth]{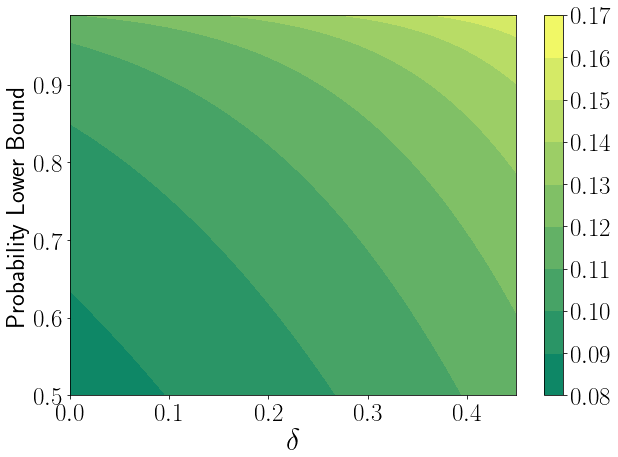}
    \caption{The upper bound derived from inequality \eqref{eq:ineq1_min}.}
    \end{subfigure} \hspace{2em}
    \begin{subfigure}{6.5cm}
    	    \includegraphics[width=\linewidth]{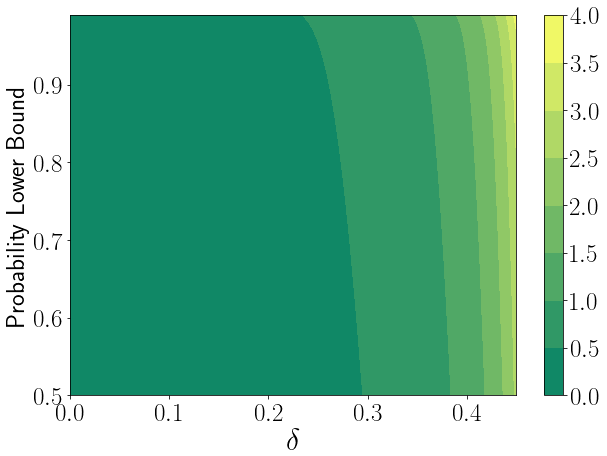}
    \caption{The upper bound derived from inequality \eqref{eq:ineq2_min}.}
    \end{subfigure}
    \caption{Comparison of the upper bounds given by Theorem~\ref{thm:global} for the distance $\|\hat X\hat X^T - M^*\|_F$ with $\hat X$ being an arbitrary local minimizer.}
    \label{fig:bound_compare}
\end{figure*}
\begin{figure*}[t]
    \centering
    \begin{subfigure}{6.5cm}
    	    \includegraphics[width=\linewidth]{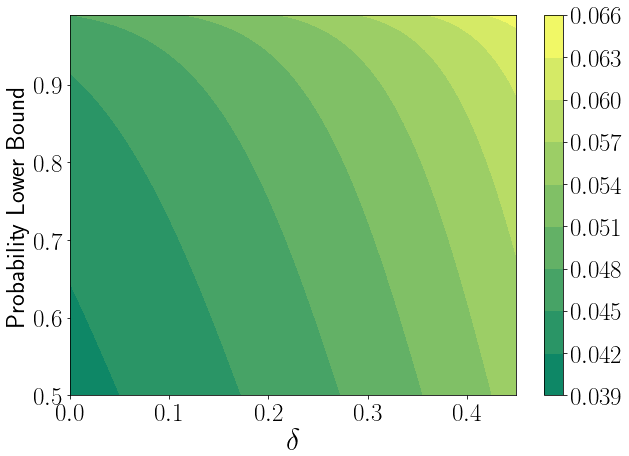}
    \caption{The upper bound derived from inequality \eqref{eq:ineq1_min}.}
    \end{subfigure} \hspace{2em}
    \begin{subfigure}{6.5cm}
    	    \includegraphics[width=\linewidth]{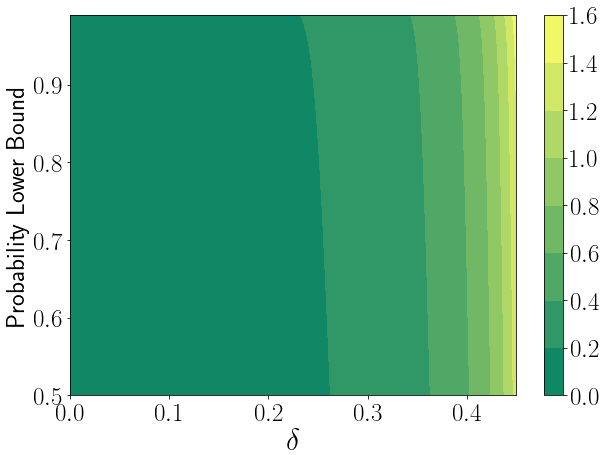}
    \caption{The upper bound derived from inequality \eqref{eq:ineq2_min}.}
    \end{subfigure}
    \caption{Upper bounds given by Theorem~\ref{thm:global} for the distance $\|\hat X\hat X^T - M^*\|_F$ based on the explicit distribution on the noise $w$.}
    \label{fig:bound_compare_iso}
\end{figure*}

\section{Numerical Illustration}\label{sec:demo}

In the next, we will empirically study the developed probabilistic guarantees and demonstrate the distance $\|\hat X\hat X^T- M^*\|_F$ between any local minimizer $\hat X$ and the ground truth $M^*$ as well as the value of the RIP constant $\delta$ required to be satisfied by the linear operator $\mathcal A$.

Before delving into the numerical illustration, note that the probability $\mathbb{P}(\norm{\mathbf A^Tw} \leq  \epsilon)$ used in both Theorem~\ref{thm:global} and Theorem~\ref{thm:local_r} can be bounded from below by the probability $\mathbb{P}(\norm{w} \leq w_0)$ with $w_0=\epsilon/\norm{\mathbf A}_2$. The latter probability can be easily estimated when the probability distribution of the noise $w$ is given. As an example, in the simplest case when $w$ is sampled from an isotropic Gaussian distribution, i.e., $w \sim \mathcal{N}(0,\sigma^2 I_m)$, the random variable $\norm{w/\sigma}^2$ follows the chi-square distribution and one can apply the Chernoff bound to obtain
\begin{equation*}
	\begin{aligned}
		\mathbb{P}(\|w\| \leq w_0) &= 1- \mathbb{P}\left(\norm*{\frac{w}{\sigma}}^2 \geq \frac{w_0^2}{\sigma^2}\right) \\
		&\geq  1 - \inf_{ 0 \leq t<1/2} (1-2t)^{-m/2} \mathrm e^{- tw_0^2/\sigma^2}.
	\end{aligned}
\end{equation*}

After solving the minimization problem in the above equation, we obtain
\[
		1 - \left(\frac{2m\sigma^2}{w_0^2}\right)^{-m/2}\mathrm e^{m-\frac{w_0^2}{2 \sigma^2}} \leq \mathbb{P}(\| w \| \leq w_0).
\]
More generally, if $w$ is a $(\sigma/\sqrt{m})$-sub-Gaussian vector, then applying Lemma 1 in \citet{jin2019short} leads to
\[
	1- 2\mathrm e^{-\frac{w_0^2}{16m \sigma^2}} \leq \mathbb{P}(\|w\| \leq w_0).
\]

\begin{figure*}[t]
    \centering
    \begin{subfigure}{6.5cm}
    	    \includegraphics[width=\linewidth]{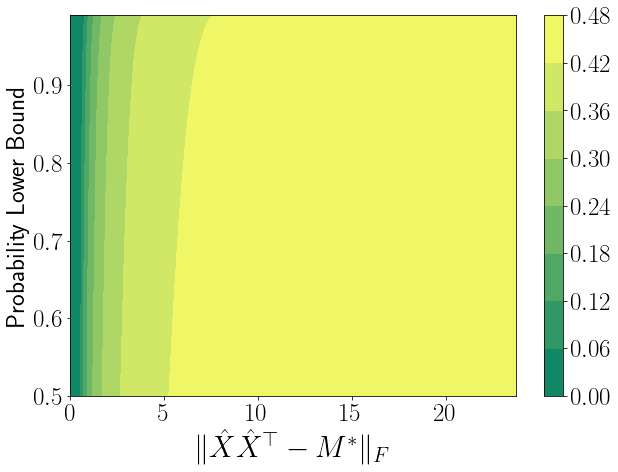}
    \caption{$\delta$ bound in Theorem~\ref{thm:global} with $\tau=+\infty$.}
    \end{subfigure}\hspace{2em}
    \begin{subfigure}{6.5cm}
    	    \includegraphics[width=\linewidth]{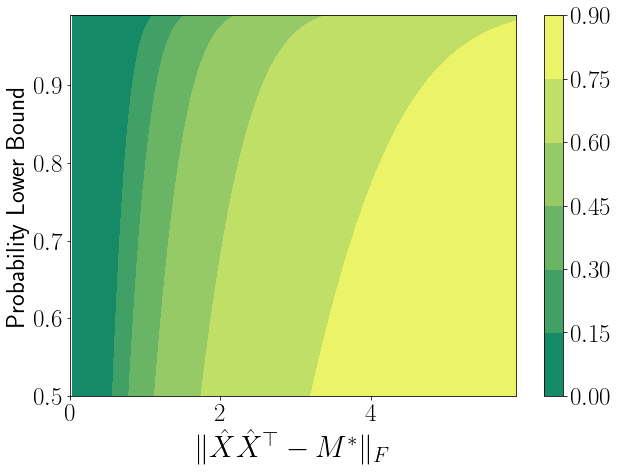}
    \caption{$\delta$ bound in Theorem~\ref{thm:local_r} with $\tau = 0.2$.}
    \end{subfigure}\vspace{2em}\\
    \begin{subfigure}{6.5cm}
    	    \includegraphics[width=\linewidth]{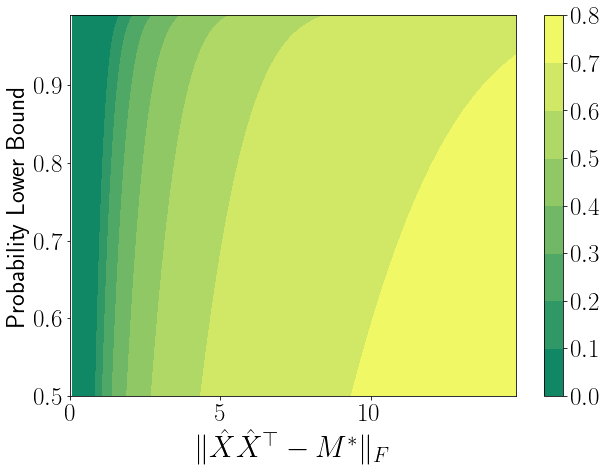}
    \caption{$\delta$ bound in Theorem~\ref{thm:local_r} with $\tau=0.5$.}
    \end{subfigure}\hspace{2em}
    \begin{subfigure}{6.5cm}
    	    \includegraphics[width=\linewidth]{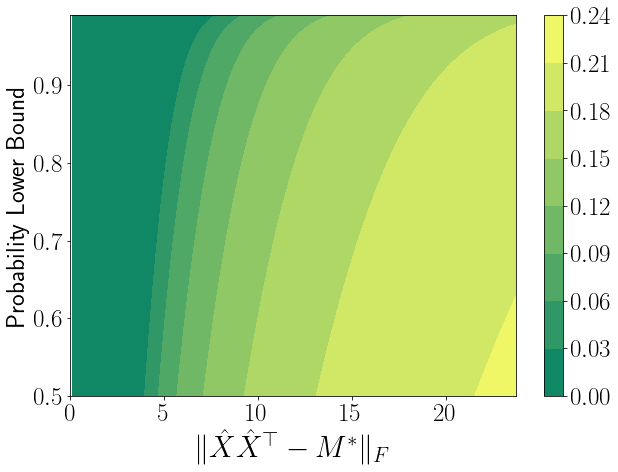}
    \caption{$\delta$ bound in Theorem~\ref{thm:local_r} with $\tau=0.8$.}
    \end{subfigure}
    \caption{Comparison of the maximum RIP constants $\delta$ allowed by Theorem~\ref{thm:global} and Theorem~\ref{thm:local_r} to guarantee a given maximum distance $\|\hat X\hat X^T - M^*\|_F$ for an arbitrary local minimizer $\hat X$ satisfying \eqref{eq:localregion} with a given probability.}
    \label{fig:thm_compare}
\end{figure*}

For numerical illustration, assume that $n=50$, $m=10$ and $\norm{\mathbf A}_2 \leq 2$, while the noise $w$ is a $(0.05/\sqrt{m})$-sub-Gaussian vector. We also assume that the ground truth $M^*$ is of rank 5 with the largest eigenvalue being 1.5 and the smallest eigenvalue being 1.

First, we explore the two inequalities \eqref{eq:ineq1_min} and \eqref{eq:ineq2_min} in Theorem~\ref{thm:global} to obtain two upper bounds on $\|\hat X\hat X^T - M^*\|_F$, where $\hat X$ denotes any arbitrary (worst) local minimizer. Figure~\ref{fig:bound_compare} gives the contour plots of the two upper bounds on $\|\hat X\hat X^T - M^*\|_F$, which hold with the given probability on the $y$-axis and the given RIP constant $\delta$ from 0 to $1/2$ on the $x$-axis. While the final bound on $\|\hat X\hat X^T - M^*\|_F$ is often determined by the inequality \eqref{eq:ineq2_min}, the inequality \eqref{eq:ineq1_min} is needed theoretically to deal with the case when $\hat X$ has a singular value close to 0.

Furthermore, a tighter bound on the success probability can be derived by calculating the exact probability $\mathbb P(\|\mathbf A^Tw\| \leq \epsilon)$ for an explicit distribution on $w$. Figure~\ref{fig:bound_compare_iso} is obtained in this way under the same assumptions as those for Figure~\ref{fig:bound_compare} except that $w$ is isotropic Gaussian with the same parameter in the sub-Gaussian assumption. Compared with Figure~\ref{fig:bound_compare}, the shape is similar, but the bound is tighter.

Next, we illustrate the bounds given by Theorem~\ref{thm:global} and Theorem~\ref{thm:local_r}. Figure~\ref{fig:thm_compare} shows the contour plots of the maximum RIP constant $\delta$ that is necessary to guarantee that each local minimizer $\hat X$ (satisfying the inequality \eqref{eq:localregion} when Theorem~\ref{thm:local_r} is applied) lies within a certain neighborhood of the ground truth (measured via the distance $\|\hat X\hat X^T - M^*\|_F$ on the $x$-axis) with a given probability on the $y$-axis, as implied by the respective global and local guarantees. Figure~\ref{fig:thm_compare} clearly shows how a smaller RIP constant $\delta$ leads to a tighter bound on the distance $\|\hat X\hat X^T - M^*\|_F$ with a higher probability. In addition, the local guarantee generally requires a looser RIP assumption as it still holds even when $\delta>1/2$. However, as the parameter $\tau$ in Theorem~\ref{thm:local_r} increases, the local bound also degrades quickly, sometimes becoming worse than the global bound as illustrated in Figure~\ref{fig:thm_compare}(d).

Moreover, in our experiment, we have also tried different values of problem parameters $m$ and $n$. They all yield similar results and are included in Appendix~\ref{app:figures} for completeness.

\section{Conclusion}

In this paper, we develop global and local analyses for the locations of the local minima of the low-rank matrix recovery problem with noisy linear measurements. Unlike the existing results, the probability distribution of the noise is arbitrary and the RIP constant of the problem is free to take any arbitrary value. The developed results encompass the state-of-the-art results on the non-existence of spurious solutions in the noiseless case. Furthermore, we prove the strict saddle property, which guarantees the global convergence of the perturbed gradient descent method in polynomial time. Our analyses show how the value of the RIP constant and the intensity of noise affect the landscape of the non-convex learning problem and the locations of the local minima relative to the ground truth. Future research directions include extending our results into the cases when the matrices are asymmetric, the measurements are nonlinear, or the overparameterized regime in which $\rank(M^*)$ is less than $r$.

\bibliography{references}

\begin{thebibliography}{29}
\providecommand{\natexlab}[1]{#1}

\bibitem[{Bhojanapalli, Neyshabur, and Srebro(2016)}]{bhojanapalli2016global}
Bhojanapalli, S.; Neyshabur, B.; and Srebro, N. 2016.
\newblock Glob\-al Optimality of Local Search for Low Rank Matrix Recovery.
\newblock In \emph{Advances in Neural Information Processing Systems},
  volume~29.

\bibitem[{Bi and Lavaei(2020)}]{bi2020global}
Bi, Y.; and Lavaei, J. 2020.
\newblock Global and Local Analyses of Nonlinear Low-Rank Matrix Recovery
  Problems.
\newblock ArXiv:2010.04349.

\bibitem[{Boumal(2016)}]{boumal2016nonconvex}
Boumal, N. 2016.
\newblock Nonconvex phase synchronization.
\newblock \emph{SIAM Journal on Optimization}, 26(4): 2355--2377.

\bibitem[{Burer and Monteiro(2003)}]{burer2003nonlinear}
Burer, S.; and Monteiro, R.~D. 2003.
\newblock A nonlinear programming algorithm for solving semidefinite programs
  via low-rank factorization.
\newblock \emph{Mathematical Programming}, 95(2): 329--357.

\bibitem[{Cand{\`e}s and Recht(2009)}]{candes2009exact}
Cand{\`e}s, E.~J.; and Recht, B. 2009.
\newblock Exact matrix completion via convex optimization.
\newblock \emph{Foundations of Computational Mathematics}, 9(6): 717--772.

\bibitem[{Cand{\`e}s and Tao(2010)}]{candes2010power}
Cand{\`e}s, E.~J.; and Tao, T. 2010.
\newblock The power of convex relaxation: Near-optimal matrix completion.
\newblock \emph{IEEE Transactions on Information Theory}, 56(5): 2053--2080.

\bibitem[{Chen and Chi(2018)}]{chen2018harnessing}
Chen, Y.; and Chi, Y. 2018.
\newblock Harnessing structures in big data via guaranteed low-rank matrix
  estimation: Recent theory and fast algorithms via convex and nonconvex
  optimization.
\newblock \emph{IEEE Signal Processing Magazine}, 35(4): 14--31.

\bibitem[{Chi, Lu, and Chen(2019)}]{chi2019nonconvex}
Chi, Y.; Lu, Y.~M.; and Chen, Y. 2019.
\newblock Nonconvex optimization meets low-rank matrix factorization: An
  overview.
\newblock \emph{IEEE Transactions on Signal Processing}, 67(20): 5239--5269.

\bibitem[{Ge, Jin, and Zheng(2017)}]{ge2017no}
Ge, R.; Jin, C.; and Zheng, Y. 2017.
\newblock No Spurious Local Minima in Nonconvex Low Rank Problems: A Unified
  Geometric Analysis.
\newblock In \emph{Proceedings of the 34th International Conference on Machine
  Learning}, volume~70 of \emph{Proceedings of Machine Learning Research},
  1233--1242.

\bibitem[{Ha, Liu, and Barber(2020)}]{HLB2020}
Ha, W.; Liu, H.; and Barber, R.~F. 2020.
\newblock An Equivalence Between Critical Points for Rank Constraints Versus
  Low-Rank Factorizations.
\newblock \emph{SIAM Journal on Optimization}, 30(4): 2927--2955.

\bibitem[{Jin et~al.(2017)Jin, Ge, Netrapalli, Kakade, and Jordan}]{JGNK2017}
Jin, C.; Ge, R.; Netrapalli, P.; Kakade, S.~M.; and Jordan, M.~I. 2017.
\newblock How to Escape Saddle Points Efficiently.
\newblock In \emph{Proceedings of the 34th International Conference on Machine
  Learning}, volume~70 of \emph{Proceedings of Machine Learning Research},
  1724--1732.

\bibitem[{Jin et~al.(2019{\natexlab{a}})Jin, Netrapalli, Ge, Kakade, and
  Jordan}]{jin2019short}
Jin, C.; Netrapalli, P.; Ge, R.; Kakade, S.~M.; and Jordan, M.~I.
  2019{\natexlab{a}}.
\newblock A short note on concentration inequalities for random vectors with
  {subGaussian} norm.
\newblock ArXiv:1902.03736.

\bibitem[{Jin et~al.(2021)Jin, Lavaei, Sojoudi, and Baldick}]{JLSB2021}
Jin, M.; Lavaei, J.; Sojoudi, S.; and Baldick, R. 2021.
\newblock Boundary Defense Against Cyber Threat for Power System State
  Estimation.
\newblock \emph{IEEE Transactions on Information Forensics and Security}, 16:
  1752--1767.

\bibitem[{Jin et~al.(2019{\natexlab{b}})Jin, Molybog, Mohammadi-Ghazi, and
  Lavaei}]{jin2019towards}
Jin, M.; Molybog, I.; Mohammadi-Ghazi, R.; and Lavaei, J. 2019{\natexlab{b}}.
\newblock Towards robust and scalable power system state estimation.
\newblock In \emph{2019 IEEE 58th Conference on Decision and Control (CDC)},
  3245--3252. IEEE.

\bibitem[{Li et~al.(2020)Li, Zhu, Man-Cho~So, and Vidal}]{li2020nonconvex}
Li, X.; Zhu, Z.; Man-Cho~So, A.; and Vidal, R. 2020.
\newblock Nonconvex robust low-rank matrix recovery.
\newblock \emph{SIAM Journal on Optimization}, 30(1): 660--686.

\bibitem[{Park et~al.(2018)Park, Kyrillidis, Caramanis, and
  Sanghavi}]{park2018finding}
Park, D.; Kyrillidis, A.; Caramanis, C.; and Sanghavi, S. 2018.
\newblock Finding low-rank solutions via nonconvex matrix factorization,
  efficiently and provably.
\newblock \emph{SIAM Journal on Imaging Sciences}, 11(4): 2165--2204.

\bibitem[{Park et~al.(2017)Park, Kyrillidis, Carmanis, and Sanghavi}]{PKCS2017}
Park, D.; Kyrillidis, A.; Carmanis, C.; and Sanghavi, S. 2017.
\newblock Non-square Matrix Sensing Without Spurious Local Minima via the
  {Burer}--{Monteiro} Approach.
\newblock In \emph{Proceedings of the 20th International Conference on
  Artificial Intelligence and Statistics}, volume~54 of \emph{Proceedings of
  Machine Learning Research}, 65--74.

\bibitem[{Recht, Fazel, and Parrilo(2010)}]{recht2010guaranteed}
Recht, B.; Fazel, M.; and Parrilo, P.~A. 2010.
\newblock Guaranteed minimum-rank solutions of linear matrix equations via
  nuclear norm minimization.
\newblock \emph{SIAM Review}, 52(3): 471--501.

\bibitem[{Shechtman et~al.(2015)Shechtman, Eldar, Cohen, Chapman, Miao, and
  Segev}]{shechtman2015phase}
Shechtman, Y.; Eldar, Y.~C.; Cohen, O.; Chapman, H.~N.; Miao, J.; and Segev, M.
  2015.
\newblock Phase retrieval with application to optical imaging: A contemporary
  overview.
\newblock \emph{IEEE Signal Processing Magazine}, 32(3): 87--109.

\bibitem[{Singer(2011)}]{singer2011angular}
Singer, A. 2011.
\newblock Angular synchronization by eigenvectors and semidefinite programming.
\newblock \emph{Applied and Computational Harmonic Analysis}, 30(1): 20--36.

\bibitem[{Tu et~al.(2016)Tu, Boczar, Simchowitz, Soltanolkotabi, and
  Recht}]{TBSS2016}
Tu, S.; Boczar, R.; Simchowitz, M.; Soltanolkotabi, M.; and Recht, B. 2016.
\newblock Low-Rank Solutions of Linear Matrix Equations via {Procrustes} Flow.
\newblock In \emph{Proceedings of the 33rd International Conference on Machine
  Learning}, volume~48 of \emph{Proceedings of Machine Learning Research},
  964--973.

\bibitem[{Wang, Zhang, and Gu(2017)}]{wang2017unified}
Wang, L.; Zhang, X.; and Gu, Q. 2017.
\newblock A unified computational and statistical framework for nonconvex
  low-rank matrix estimation.
\newblock In \emph{Proceedings of the 20th International Conference on
  Artificial Intelligence and Statistics}, volume~54 of \emph{Proceedings of
  Machine Learning Research}, 981--990.

\bibitem[{Zhang and Zhang(2020)}]{zhang2020many}
Zhang, G.; and Zhang, R.~Y. 2020.
\newblock How Many Samples Is a Good Initial Point Worth in Low-Rank Matrix
  Recovery?
\newblock In \emph{Advances in Neural Information Processing Systems},
  volume~33, 12583--12592.

\bibitem[{Zhang(2021)}]{Zhang2021-p}
Zhang, R.~Y. 2021.
\newblock Sharp Global Guarantees for Nonconvex Low-Rank Matrix Recovery in the
  Overparameterized Regime.
\newblock ArXiv:2104.10790.

\bibitem[{Zhang et~al.(2018{\natexlab{a}})Zhang, Josz, Sojoudi, and
  Lavaei}]{zhang2018much}
Zhang, R.~Y.; Josz, C.; Sojoudi, S.; and Lavaei, J. 2018{\natexlab{a}}.
\newblock How Much Restricted Isometry Is Needed in Nonconvex Matrix Recovery?
\newblock In \emph{Advances in Neural Information Processing Systems},
  volume~31.

\bibitem[{Zhang, Sojoudi, and Lavaei(2019)}]{zhang2019sharp}
Zhang, R.~Y.; Sojoudi, S.; and Lavaei, J. 2019.
\newblock Sharp Restricted Isometry Bounds for the Inexistence of Spurious
  Local Minima in Nonconvex Matrix Recovery.
\newblock \emph{Journal of Machine Learning Research}, 20(114): 1--34.

\bibitem[{Zhang et~al.(2018{\natexlab{b}})Zhang, Wang, Yu, and Gu}]{ZWYG2018}
Zhang, X.; Wang, L.; Yu, Y.; and Gu, Q. 2018{\natexlab{b}}.
\newblock A Primal-Dual Analysis of Global Optimality in Nonconvex Low-Rank
  Matrix Recovery.
\newblock In \emph{Proceedings of the 35th International Conference on Machine
  Learning}, volume~80 of \emph{Proceedings of Machine Learning Research},
  5862--5871.

\bibitem[{Zhu et~al.(2018)Zhu, Li, Tang, and Wakin}]{ZLTW2018}
Zhu, Z.; Li, Q.; Tang, G.; and Wakin, M.~B. 2018.
\newblock Global Optimality in Low-Rank Matrix Optimization.
\newblock \emph{IEEE Transactions on Signal Processing}, 66(13): 3614--3628.

\bibitem[{Zhu et~al.(2021)Zhu, Li, Tang, and Wakin}]{zhu2021global}
Zhu, Z.; Li, Q.; Tang, G.; and Wakin, M.~B. 2021.
\newblock The global optimization geometry of low-rank matrix optimization.
\newblock \emph{IEEE Transactions on Information Theory}, 67(2): 1308--1331.

\end{thebibliography}

\section*{Acknowledgments}

This work was supported by grants from AFOSR, ARO, ONR, and NSF.

\newpage
\onecolumn
\appendix

\section{Remark on Theorem~\ref{thm:global}}\label{app:remark}

The two upper bounds on the distance $\norm{\hat X\hat X^T-M^*}_F$ can be obtained for any local minimizer $\hat X$ by solving the two quadratic-like inequalities \eqref{eq:ineq1_min} and \eqref{eq:ineq2_min}, and the larger bound needs to be used. To be explicit,
\[
\norm{\hat X\hat X^T-M^*}_F \leq \max\{T_1,T_2\},
\]
with
\begin{align*}
T_1&=\frac{\epsilon \sqrt{r} + \sqrt{r \epsilon^2 + 16(1-\delta)\epsilon \sqrt{r}}}{2(1-\delta)}, \\
T_2&=\left(\frac{2\sqrt{6\epsilon (1+\delta)^3}+ 3\sqrt{8\epsilon(1+\delta)^3+\frac{8}{3}(1-2\delta)(1+\delta)(2\epsilon r + 2\sqrt{2\epsilon(1+\delta)}\|M^*\|_F^{1/2})}}{4(1-2\delta)}\right)^2.
\end{align*}

\section{Proofs of Lemma~\ref{lem:case2} and Theorem~\ref{thm:global_1}}\label{app:case2proof}

\begin{proof}[Proof of Lemma~\ref{lem:case2}]
Let $Z \in \RR^{n \times r}$ be a matrix satisfying $ZZ^T=M^*$. Similar to the proof of Theorem~\ref{thm:local_r}, we introduce an optimization problem as follows:
\begin{equation}\label{eq:etaopt}
\begin{aligned}
\max_{\eta,\hat{\mathbf H}} \quad & \eta \\
\st \quad & \norm{\hat{\mathbf X}^T\hat{\mathbf H}\mathbf e} \leq 2\norm{\hat X}_2\epsilon+\kappa, \\
& 2I_r \otimes \mat_S(\hat{\mathbf H}\mathbf e)+\hat{\mathbf X}^T\hat{\mathbf X} \succeq -(2\epsilon+\kappa)I_{nr}, \\
& \eta I_{n^2} \preceq \hat{\mathbf H} \preceq I_{n^2},
\end{aligned}
\end{equation}
where its optimal value $\eta^*(\hat X)$ satisfies the inequality
\begin{equation}\label{eq:etaineq}
\eta^*(\hat X) \geq \frac{1-\delta^*(\hat X)}{1+\delta^*(\hat X)} \geq \frac{1-\delta}{1+\delta}.
\end{equation}
In the remaining part, we will prove the following upper bound on $\eta^*(\hat X)$:
\begin{equation}\label{eq:etaupper}
\eta^*(\hat X) \leq \frac{1}{3}+\frac{(2\epsilon+\kappa)\sqrt r+\sqrt{2\kappa(1+\delta)}+2\sqrt{2\epsilon(1+\delta)}\norm{\hat X}_2}{\norm{\mathbf e}}.
\end{equation}
The inequality \eqref{eq:ineq2} is a consequence of \eqref{eq:etaineq}, \eqref{eq:etaupper} and the inequality
\[
\norm{\hat X}_2 \leq \norm{\hat X\hat X^T}_F^{1/2} \leq \norm{\hat X\hat X^T-M^*}_F^{1/2}+\norm{M^*}_F^{1/2}.
\]

The proof of the upper bound \eqref{eq:etaupper} can be completed by finding a feasible solution to the dual problem of \eqref{eq:etaopt}:
\begin{equation}\label{eq:etaoptdual}
\begin{aligned}
\min_{\substack{U_1,U_2,W, \\ G,\lambda,y}} \quad & \tr(U_2)+\langle\hat{\mathbf X}^T\hat{\mathbf X},W\rangle+(2\epsilon+\kappa)\tr(W)+(2\norm{\hat X}_2\epsilon+\kappa)^2\lambda+\tr(G) \\
\st \quad & \tr(U_1)=1, \\
& (\hat{\mathbf X}y-w)\mathbf e^T+\mathbf e(\hat{\mathbf X}y-w)^T=U_1-U_2, \\
& \begin{bmatrix}
G & -y \\
-y^T & \lambda
\end{bmatrix} \succeq 0, \\
& U_1 \succeq 0, \quad U_2 \succeq 0, \quad W=\begin{bmatrix}
W_{1,1} & \cdots & W_{r,1}^T \\
\vdots & \ddots & \vdots \\
W_{r,1} & \cdots & W_{r,r}
\end{bmatrix} \succeq 0, \\
& w=\sum_{i=1}^r\vect(W_{i,i}).
\end{aligned}
\end{equation}

Before describing the choice of the dual feasible solution, we need to represent the error vector $\mathbf e$ in a different form. Let $\mathcal P \in \RR^{n \times n}$ be the orthogonal projection matrix onto the range of $\hat X$, and $\mathcal P_\perp \in \RR^{n \times n}$ be the orthogonal projection matrix onto the orthogonal complement of the range of $\hat X$. Then, $Z$ can be decomposed as $Z=\mathcal PZ+\mathcal P_\perp Z$, and there exists a matrix $R \in \RR^{r \times r}$ such that $\mathcal PZ=\hat XR$. Note that
\[
ZZ^T=\mathcal PZZ^T\mathcal P+\mathcal PZZ^T\mathcal P_\perp+\mathcal P_\perp ZZ^T\mathcal P+\mathcal P_\perp ZZ^T\mathcal P_\perp.
\]
Thus, if we choose
\begin{equation}\label{eq:ydef}
\hat Y=\frac{1}{2}\hat X-\frac{1}{2}\hat XRR^T-\mathcal P_\perp ZR^T, \quad \hat y=\vect(\hat Y),
\end{equation}
then it can be verified that
\begin{gather*}
\hat X\hat Y^T+\hat Y\hat X^T-\mathcal P_\perp ZZ^T\mathcal P_\perp=\hat X\hat X^T-ZZ^T, \\
\langle \hat X\hat Y^T+\hat Y\hat X^T,\mathcal P_\perp ZZ^T\mathcal P_\perp\rangle=0.
\end{gather*}
Moreover, we have
\begin{equation}\label{eq:ynorm}
\begin{aligned}
\norm{\hat X\hat Y^T+\hat Y\hat X^T}_F^2&=2\tr(\hat X^T\hat X\hat Y^T\hat Y)+\tr(\hat X^T\hat Y\hat X^T\hat Y)+\tr(\hat Y^T\hat X\hat Y^T\hat X) \\
&\geq 2\tr(\hat X^T\hat X\hat Y^T\hat Y) \geq 2\sigma_r(\hat X)^2\norm{\hat Y}_F^2,
\end{aligned}
\end{equation}
in which the first inequality is due to
\[
\tr(\hat X^T\hat Y\hat X^T\hat Y)=\frac{1}{4}\tr((\hat X^T\hat X(I_r-RR^T))^2)=\frac{1}{4}\tr((\hat X(I_r-RR^T)\hat X^T)^2) \geq 0.
\]

Assume first that $Z_\perp=\mathcal P_\perp Z \neq 0$. The other case will be handled at the end of this proof. In the case when $Z_\perp \neq 0$, we also have $\hat X\hat Y^T+\hat Y\hat X^T \neq 0$. Otherwise, the inequality \eqref{eq:ynorm} and the assumption $\sigma_r(\hat X)>0$ imply that $\hat Y=0$. The orthogonality and the definition of $\hat Y$ in \eqref{eq:ydef} then give rise to
\[
\hat X-\hat XRR^T=0, \quad \mathcal P_\perp ZR^T=0.
\]
The first equation above implies that $R$ is invertible since $\hat X$ has full column rank, which contradicts $Z_\perp \neq 0$. Now, define the unit vectors
\[
\hat u_1=\frac{\hat{\mathbf X}\hat y}{\norm{\hat{\mathbf X}\hat y}}, \quad \hat u_2=\frac{\vect(Z_\perp Z_\perp^T)}{\norm{Z_\perp Z_\perp^T}_F}.
\]
Then, $\hat u_1 \perp \hat u_2$ and
\begin{equation}\label{eq:epolar}
\mathbf e=\norm{\mathbf e}(\sqrt{1-\alpha^2}\hat u_1-\alpha\hat u_2)
\end{equation}
with
\begin{equation}\label{eq:alphadef}
\alpha=\frac{\norm{Z_\perp Z_\perp^T}_F}{\norm{\hat X\hat X^T-ZZ^T}_F}.
\end{equation}

We first describe our choices of the dual variables $W$ and $y$ (which will be scaled later). Let
\[
\hat X^T\hat X=QSQ^T, \quad Z_\perp Z_\perp^T=PGP^T,
\]
with orthogonal matrices $Q,P$ and diagonal matrices $S,G$ such that $S_{11}=\sigma_r(\hat X)^2$. Fix a constant $\gamma \in [0,1]$ that is to be determined and define
\begin{gather*}
V_i=k^{1/2}G_{ii}^{1/2}PE_{i1}Q^T, \quad \forall i=1,\dots,r, \\
W=\sum_{i=1}^r\vect(V_i)\vect(V_i)^T, \quad y=l\hat y,
\end{gather*}
with $\hat y$ defined in \eqref{eq:ydef} and
\[
k=\frac{\gamma}{\norm{\mathbf e}\norm{Z_\perp Z_\perp^T}_F}, \quad l=\frac{\sqrt{1-\gamma^2}}{\norm{\mathbf e}\norm{\hat{\mathbf X}\hat y}}.
\]
Here, $E_{ij}$ is the elementary matrix of size $n \times r$ with the $(i,j)$-entry being $1$. By our construction, $\hat X^TV_i=0$, which implies that
\begin{equation}\label{eq:dualobj1}
\langle\hat{\mathbf X}^T\hat{\mathbf X},W\rangle=\sum_{i=1}^r\norm{\hat XV_i^T+V_i\hat X^T}_F^2=2\sum_{i=1}^r\tr(\hat X^T\hat XV_i^TV_i)=2k\sigma_r(\hat X)^2\sum_{i=1}^rG_{ii}=2\beta\gamma,
\end{equation}
with
\begin{equation}\label{eq:betadef}
\beta=\frac{\sigma_r(\hat X)^2\tr(Z_\perp Z_\perp^T)}{\norm{\hat X\hat X^T-ZZ^T}_F\norm{Z_\perp Z_\perp^T}_F}.
\end{equation}
In addition,
\begin{equation}\label{eq:dualobj2}
\tr(W)=\sum_{i=1}^r\norm{V_i}_F^2=k\sum_{i=1}^rG_{ii}=k\tr(Z_\perp Z_\perp^T) \leq \frac{\sqrt r}{\norm{\mathbf e}},
\end{equation}
and
\[
w=\sum_{i=1}^r\vect(W_{i,i})=\sum_{i=1}^rV_iV_i^T=kZ_\perp Z_\perp^T.
\]
Therefore,
\[
\hat{\mathbf X}y-w=\frac{1}{\norm{\mathbf e}}(\sqrt{1-\gamma^2}\hat u_1-\gamma\hat u_2),
\]
which together with \eqref{eq:epolar} implies that
\begin{equation}\label{eq:dualobj3}
\norm{\mathbf e}\norm{\hat{\mathbf X}y-w}=1, \quad \langle\mathbf e,\hat{\mathbf X}y-w\rangle=\gamma\alpha+\sqrt{1-\gamma^2}\sqrt{1-\alpha^2}=\psi(\gamma).
\end{equation}
Next, the inequality \eqref{eq:ynorm} and the assumption on $\sigma_r(\hat X)$ imply that
\begin{equation}\label{eq:dualobj4}
\epsilon\norm{y} \leq \frac{\sqrt{1-\gamma^2}\epsilon}{\sqrt 2\sigma_r(\hat X)\norm{\mathbf e}} \leq \frac{\sqrt{1+\delta}\epsilon}{\sqrt{2(\epsilon+\kappa)}\norm{\mathbf e}} \leq \frac{\sqrt{\epsilon(1+\delta)}}{\sqrt 2\norm{\mathbf e}}
\end{equation}
and similarly
\begin{equation}\label{eq:dualobj5}
\kappa\norm{y} \leq \frac{\sqrt{\kappa(1+\delta)}}{\sqrt 2\norm{\mathbf e}}.
\end{equation}

Define
\[
M=(\hat{\mathbf X}y-w)\mathbf e^T+\mathbf e(\hat{\mathbf X}y-w)^T,
\]
and decompose $M$ as $M=[M]_+-[M]_-$ in which both $[M]_+ \succeq 0$ and $[M]_- \succeq 0$. Let $\theta$ be the angle between $\mathbf e$ and $\hat{\mathbf X}y-w$. By Lemma~\ref{lem:richard_14}, we have
\[
\tr([M]_+)=\norm{\mathbf e}\norm{\hat{\mathbf X}y-w}(1+\cos\theta), \quad \tr([M]_-)=\norm{\mathbf e}\norm{\hat{\mathbf X}y-w}(1-\cos\theta).
\]
Now, one can verify that $(U_1^*,U_2^*,W^*,G^*,\lambda^*,y^*)$ defined as
\begin{gather*}
U_1^*=\frac{[M]_+}{\tr([M]_+)}, \quad U_2^*=\frac{[M]_-}{\tr([M]_+)}, \quad y^*=\frac{y}{\tr([M]_+)}, \\
W^*=\frac{W}{\tr([M]_+)}, \quad \lambda^*=\frac{\norm{y^*}}{2\norm{\hat X}_2\epsilon+\kappa}, \quad G^*=\frac{1}{\lambda^*}y^*y^{*T}
\end{gather*}
forms a feasible solution to the dual problem \eqref{eq:etaoptdual} whose objective value is equal to
\[
\frac{\tr([M]_-)+\langle\hat{\mathbf X}^T\hat{\mathbf X},W\rangle+(2\epsilon+\kappa)\tr(W)+2(2\norm{\hat X}_2\epsilon+\kappa)\norm{y}}{\tr([M]_+)}.
\]
Substituting \eqref{eq:dualobj1}, \eqref{eq:dualobj2}, \eqref{eq:dualobj3}, \eqref{eq:dualobj4} and \eqref{eq:dualobj5} into the above equation, we obtain
\begin{align*}
\eta^*(\hat X) &\leq \frac{2\beta\gamma+1-\psi(\gamma)+((2\epsilon+\kappa)\sqrt r+\sqrt{2\kappa(1+\delta)}+2\sqrt{2\epsilon(1+\delta)}\norm{\hat X}_2)/\norm{\mathbf e}}{1+\psi(\gamma)} \\
&\leq \frac{2\beta\gamma+1-\psi(\gamma)}{1+\psi(\gamma)}+\frac{(2\epsilon+\kappa)\sqrt r+\sqrt{2\kappa(1+\delta)}+2\sqrt{2\epsilon(1+\delta)}\norm{\hat X}_2}{\norm{\mathbf e}}.
\end{align*}
Choosing the best value of the parameter $\gamma \in [0,1]$ to minimize the far right-side of the above inequality leads to
\[
\frac{2\beta\gamma+1-\psi(\gamma)}{1+\psi(\gamma)} \leq \eta_0(\hat X),
\]
with
\[
\eta_0(\hat X):=\begin{dcases*}
\frac{1-\sqrt{1-\alpha^2}}{1+\sqrt{1-\alpha^2}}, & if $\beta \geq \dfrac{\alpha}{1+\sqrt{1-\alpha^2}}$, \\
\frac{\beta(\alpha-\beta)}{1-\beta\alpha}, & if $\beta \leq \dfrac{\alpha}{1+\sqrt{1-\alpha^2}}$.
\end{dcases*}
\]
Here, $\alpha$ and $\beta$ are defined in \eqref{eq:alphadef} and \eqref{eq:betadef}, respectively. In the proof of Theorem~1.2 in \citet{Zhang2021-p}, it is shown that $\eta_0(\hat X) \leq 1/3$ for every $\hat X$ with $\hat X\hat X^T \neq ZZ^T$, which gives the upper bound \eqref{eq:etaupper}.

Finally, we still need to deal with the case when $\mathcal P_\perp Z=0$. In this case, we know that $\hat{\mathbf X}\hat y=\mathbf e$ with $\hat y$ defined in \eqref{eq:ydef}. Then, it is easy to check that $(U_1^*,U_2^*,W^*,G^*,\lambda^*,y^*)$ defined as
\begin{gather*}
U_1^*=\frac{\mathbf e\mathbf e^T}{\norm{\mathbf e}^2}, \quad U_2^*=0, \quad y^*=\frac{\hat y}{2\norm{\mathbf e}^2}, \\
W^*=0, \quad \lambda^*=\frac{\norm{y^*}}{2\norm{\hat X}_2\epsilon+\kappa}, \quad G^*=\frac{1}{\lambda^*}y^*y^{*T}
\end{gather*}
forms a feasible solution to the dual problem \eqref{eq:etaoptdual} whose objective value is $2(2\norm{\hat X}_2\epsilon+\kappa)\norm{y^*}$. By the inequality \eqref{eq:ynorm}, we have
\[
\eta^*(\hat X) \leq 2(2\norm{\hat X}_2\epsilon+\kappa)\norm{y^*} \leq \frac{\kappa/\sqrt 2+\sqrt 2\epsilon\norm{\hat X}_2}{\sigma_r(\hat X)\norm{\mathbf e}} \leq \frac{\sqrt{\kappa(1+\delta)/2}+\sqrt{2\epsilon(1+\delta)}\norm{\hat X}_2}{\norm{\mathbf e}}.
\]
Hence, the upper bound \eqref{eq:etaupper} still holds in this case.
\end{proof}

\begin{proof}[Proof of Theorem~\ref{thm:global_1}]
The proof of Theorem~\ref{thm:global_1} is similar to the above proof of Lemma~\ref{lem:case2} in the situation with $\kappa=0$, and we will only emphasize the difference here. In the case when $\hat X \neq 0$, after constructing the feasible solution to the dual problem \eqref{eq:etaoptdual}, we have
\begin{equation}\label{eq:global1dual}
\frac{1-\delta}{1+\delta} \leq \eta^*(\hat X) \leq \frac{\tr([M]_-)+\langle\hat{\mathbf X}^T\hat{\mathbf X},W\rangle+2\epsilon\tr(W)+4\norm{\hat X}_2\epsilon\norm{y}}{\tr([M]_+)}.
\end{equation}
Note that in the rank-1 case, one can write $\sigma_r(\hat X)=\norm{\hat X}_2$ and
\[
\norm{y} \leq \frac{\norm{\hat y}}{\norm{\mathbf e}\norm{\hat{\mathbf X}\hat y}} \leq \frac{1}{\sqrt 2\norm{\hat X}_2\norm{\mathbf e}},
\]
in which the last inequality is due to \eqref{eq:ynorm}. Substituting \eqref{eq:dualobj1}, \eqref{eq:dualobj2}, \eqref{eq:dualobj3} and the above inequality into \eqref{eq:global1dual} and choosing an appropriate $\gamma$ as shown in the proof of Lemma~\ref{lem:case2}, we obtain
\begin{align*}
\frac{1-\delta}{1+\delta} \leq \eta^*(\hat X) &\leq \frac{2\beta\gamma+1-\psi(\gamma)+(2\epsilon+2\sqrt{2}\epsilon)/\norm{\mathbf e}}{1+\psi(\gamma)} \\
&\leq \frac{1}{3}+\frac{2\epsilon+2\sqrt{2}\epsilon}{\norm{\mathbf e}},
\end{align*}
which implies inequality \eqref{eq:ineqrank1} under the probabilistic event that $\norm{q} \leq \epsilon$.

In the case when $\hat X=0$, $(U_1^*,U_2^*,W^*,G^*,\lambda^*,y^*)$ defined as
\begin{gather*}
U_1^*=\frac{\mathbf e\mathbf e^T}{\norm{\mathbf e}^2}, \quad U_2^*=0, \quad y^*=0, \\
W^*=\frac{ZZ^T}{2\norm{\mathbf e}^2}, \quad \lambda^*=0, \quad G^*=0
\end{gather*}
forms a feasible solution to the dual problem \eqref{eq:etaoptdual}, which shows that
\[
\frac{1-\delta}{1+\delta} \leq \eta^*(\hat X) \leq \frac{\epsilon}{\norm{\mathbf e}}.
\]
The above inequality also implies inequality \eqref{eq:ineqrank1} under the probabilistic event that $\norm{q} \leq \epsilon$.
\end{proof}

\newpage
\section{Additional Numerical Illustration}\label{app:figures}

\begin{figure}[!h]
    \centering
    \begin{subfigure}{6.5cm}
    	    \includegraphics[width=\linewidth]{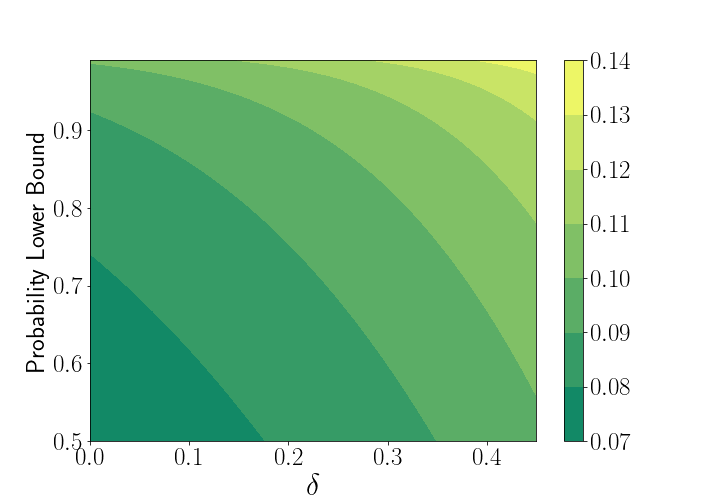}
    \caption{The upper bound derived from inequality \eqref{eq:ineq1_min}, $m=20,n=60$.}
    \end{subfigure} \hspace{2em}
    \begin{subfigure}{6.5cm}
    	    \includegraphics[width=\linewidth]{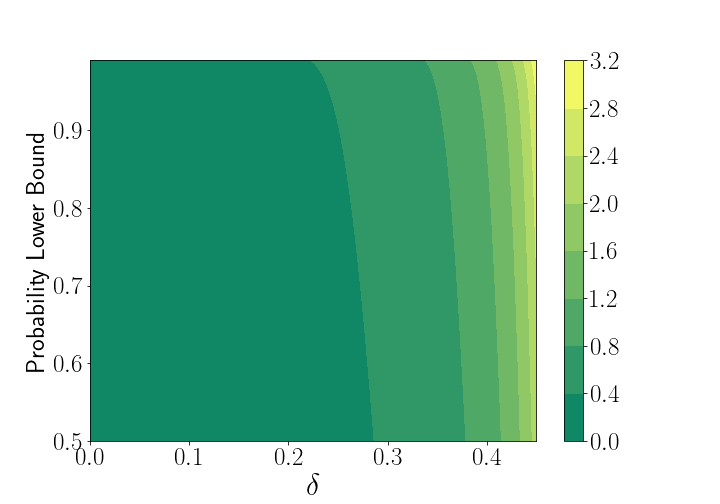}
    \caption{The upper bound derived from inequality \eqref{eq:ineq2_min}, $m=20,n=60$.}
    \end{subfigure}\vspace{2em}\\
    \begin{subfigure}{6.5cm}
    	    \includegraphics[width=\linewidth]{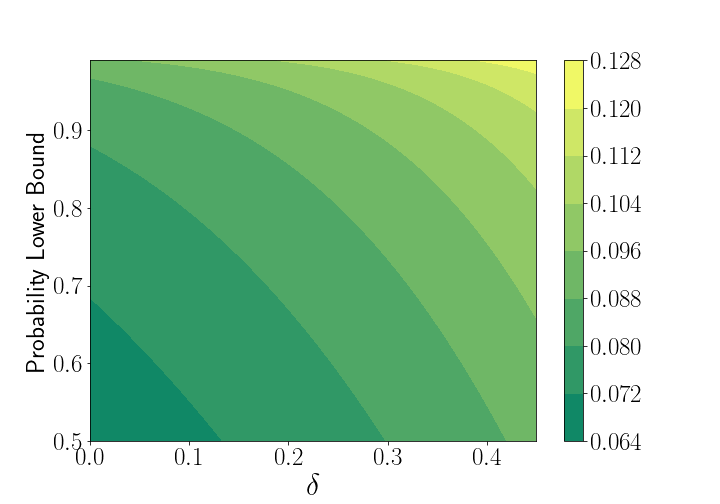}
    \caption{The upper bound derived from inequality \eqref{eq:ineq1_min}, $m=30,n=60$.}
    \end{subfigure} \hspace{2em}
    \begin{subfigure}{6.5cm}
    	    \includegraphics[width=\linewidth]{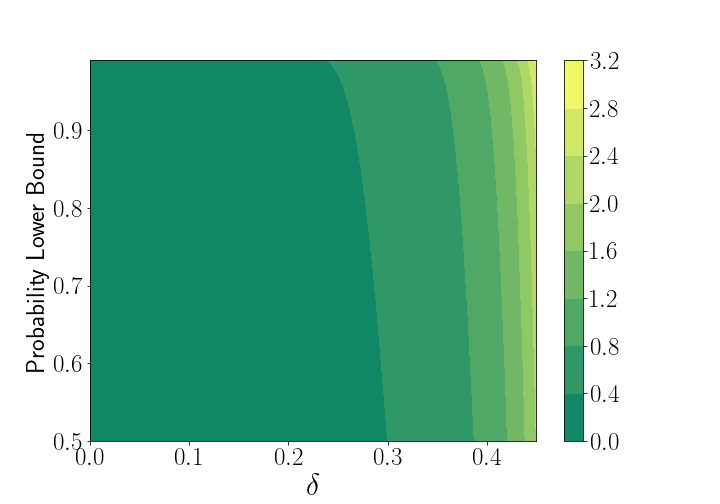}
    \caption{The upper bound derived from inequality \eqref{eq:ineq2_min}, $m=30,n=60$.}
    \end{subfigure}\vspace{2em}\\
    \begin{subfigure}{6.5cm}
    	    \includegraphics[width=\linewidth]{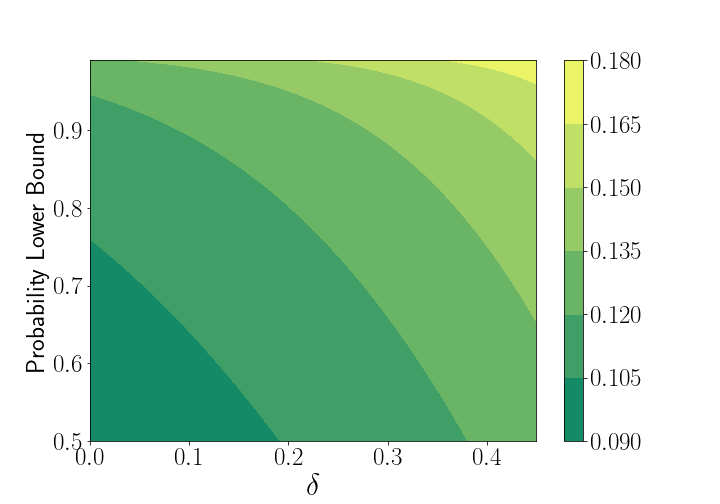}
    \caption{The upper bound derived from inequality \eqref{eq:ineq1_min}, $m=20,n=90$.}
    \end{subfigure} \hspace{2em}
    \begin{subfigure}{6.5cm}
    	    \includegraphics[width=\linewidth]{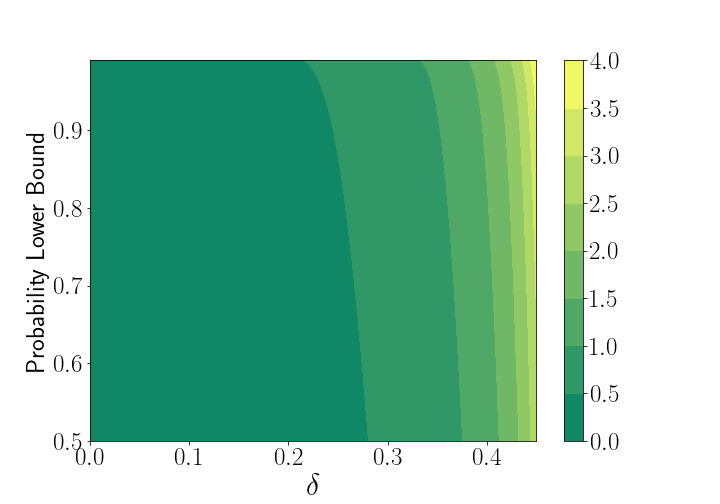}
    \caption{The upper bound derived from inequality \eqref{eq:ineq2_min}, $m=20,n=90$.}
    \end{subfigure}
    \caption{Comparison of the upper bounds given by Theorem~\ref{thm:global} for the distance $\|\hat X\hat X^T - M^*\|_F$ with $\hat X$ being an arbitrary local minimizer under varying values of problem parameters $m$ and $n$.}
\end{figure}
\end{document}